\documentclass[12pt]{amsart}
\usepackage{amscd}
\usepackage{mathrsfs}
\usepackage{a4wide}
\usepackage{color}
\usepackage{amssymb,amsmath,amsthm,amsfonts,latexsym}
\usepackage[utf8x]{inputenc}

\usepackage{amsmath,amsfonts,amssymb,amsthm} 
\usepackage{tikz}
\usepackage{hyperref}
\usepackage[all]{xy}
\newtheorem{theorem}{Theorem}[section]
\newtheorem{corollary}[theorem]{Corollary}
\newtheorem{lemma}[theorem]{Lemma}

\newtheorem{proposition}[theorem]{Proposition}

\newtheorem{question}[theorem]{Question}

\newtheorem*{claim*}{Claim}

\def\leukfrac#1/#2{\leavevmode
               \kern.1em
                \raise.9ex\hbox{\the\scriptfont0 ${}_#1$}
                \hskip -1pt\kern-.1em
                /\kern-.15em\lower.10ex\hbox{\the\scriptfont0 ${}_#2$}}

\makeatletter
\def\diam{\mathop{\operator@font diam}\nolimits}
\makeatother

\title{$\sigma$-lacunary actions of Polish groups}

\author[Greb\'{\i}k]{Jan Greb\'{\i}k${}^{1,}$ ${}^{2, }$ ${}^{3}$}

\address{${}^1$ Mathematics Institute\\
University of Warwick\\
 Coventry \\
 CV4 7AL, United Kingdom}
 \email{jan.grebik@warwick.ac.uk}

\address{${}^2$ Institute of Mathematics of the Czech Academy of Sciences\\
\v Zitn\'a 609/25  \\
 110 00 Praha 1-Nov\'e M\v esto, Czech Republic}

\email{grebikj@gmail.com}

\address{${}^3$ Department of Algebra\\
MFF UK\\
Sokolovsk\' a 83  \\
 186 00 Praha 8, Czech Republic}

\email{grebikj@gmail.com}

\thanks{{The author was supported by the GACR project 17-33849L and RVO:
67985840. The research was conducted during the author's visit at Cornell University that was partially funded by the grant GAUK~900119 of Charles University.}}

\begin{document}

\begin{abstract}
We show that every essentially countable orbit equivalence relation induced by a continuous action of a Polish group on a Polish space is $\sigma$-lacunary.
In combination with \cite{Gao Jackson} we obtain a straightforward proof of the result from \cite{Ding Gao} that every essentially countable equivalence relation that is induced by an action of abelian non-archimedean Polish group is Borel reducible to $\mathbb{E}_0$, i.e., it is essentially hyperfinite.
%Another consequence is that the dichotomy from \cite{Ben Lacunary} characterizes Borel equivalence relations induced by cli Polish groups that are essentially countable.
\end{abstract}

\maketitle

%In this paper we show the equivalence of various concepts that describes how and to what extent an orbit equivalence relation induced by a continuous action of a Polish group on a Polish space is comparable with countable Borel equivalence relations.
%First we  introduce the concepts and describe the motivation.

We say that an equivalence relation $E$ on a Polish space $X$ is {\it Borel reducible} to an equivalence relation $F$ on a Polish space $Y$, and write $E\le_B F$, if there is a Borel map $\psi:X\to Y$ such that 
$$(x,y)\in E \ \Leftrightarrow \ (\psi(x),\psi(y))\in F$$
for every $x,y\in X$.
An equivalence relation $F$ on $Y$ is {\it countable} if $|[y]_F|\le \aleph_0$ for every $y\in Y$.
We follow \cite{Kec CTBL} and say that an equivalence relation $E$ on a Polish space $X$
\begin{itemize}
	\item [(A)] is {\it essentially countable} if there is a countable Borel equivalence relation $F$ on some Polish space $Y$ such that $E\sim_B F$, i.e., $E\le_B F$ and $F\le_B E$,
	\item [(B)] admits a Borel {\it countable complete section} if there is a Borel set $B\subseteq X$ such that $[B]_{E}=X$ and $|B\cap [x]_{E}|\le \aleph_0$ for every $x\in X$. 
\end{itemize}
If we assume that $E$ is a Borel equivalence relation, then (B) $\Rightarrow$ (A) by the Lusin--Novikov Theorem, see~\cite[Theorem~18.10]{Kec}.
%Moreover, if we assume that $E$ is induced by a continuous action of a Polish group, then we have (A) $\Rightarrow$ (B), see~\cite[Lemma~3.7]{Kec Tarski}.

Let $G\curvearrowright X$ be a continuous action of a Polish group $G$ on a Polish space $X$.
We denote as $E^X_G$ the orbit equivalence relation defined as $(x,y)\in E^X_G \ \Leftrightarrow \ (\exists g\in G) \ g\cdot x=y$.
If we have such an action, then we say that $X$ is a Polish $G$-space.
It follows from \cite[Theorem~3.6]{Kec CTBL} that if $E^X_G$ satisfies (A), then $E^X_G$ satisfies (B).
It is natural to ask if we can find a Borel countable complete section with additional properties.
Following \cite{Ben Lacunary} we say that $E^X_G$ is 
\begin{itemize}
	\item [(C)] {\it $\sigma$-lacunary} if there are sequences of Borel sets $\{B_n\}_{n<\omega}$ and $\{V_n\}_{n<\omega}$ such that $\bigcup_{n<\omega} B_n$ is a countable complete section of $E^X_G$, $V_n\subseteq G$ is an open neighbourhood of $1_G$ and $B_n$ is $V_n$-lacunary for every $n\in \mathbb{N}$, i.e., if $g\cdot x=y$ for some $g\in V_n$ and $x,y\in B_n$, then $x=y$.
\end{itemize}
It follows from \cite{Kec Lacunary} that in the case when $G$ is a locally compact Polish group, then (A) and (C) are equivalent.
Main result of this paper is the following statement.

\begin{theorem}\label{th:main1}
Let $G$ be a Polish group, $X$ be a Polish $G$-space and suppose that $E^X_G$ is essentially countable.
Then $E^X_G$ is $\sigma$-lacunary.
\end{theorem}

There are some other similar concepts in the literature.
Following \cite{Kec CTBL}, we say that
\begin{itemize}
	\item [(D)] $E^X_G$ is {\it reducible to countable} if there is a countable Borel equivalence relation $F$ on some Polish space $Y$ such that $E^X_G\le_B F$,
	\item[(E)] $E^X_G$ admits {\it countable invariants} if there is a Polish space $Y$ and a Borel map $\varphi:X\to Y$ such that $|\varphi([x]_{E^X_G})|\le \aleph_0$ for every $x\in X$ and $\varphi([x]_{E^X_G})\cap \varphi([y]_{E^X_G})=\emptyset$ whenever $(x,y)\not\in E^X_G$ (see~\cite[Section~7.6.]{Kanovei}).
\end{itemize}
Next we summarize what is known about these concepts.
It is easy to see that (A) $\Rightarrow$ (D) $\Rightarrow$ (E) and (C) $\Rightarrow$ (B).
Moreover, $(A)$ and $(D)$ implies that $E^X_G$ is Borel.
If we suppose that $E^X_G$ is a Borel equivalence relation, then (E) $\Rightarrow$ (D) by \cite[Lemma~7.6.1.]{Kanovei}.
Altogether, combination of Theorem~\ref{th:main1} and the discussion above yields that, if $E^X_G$ is Borel, then all the concepts are equivalent.

In fact, we show that if $E^X_G$ satisfies (E), then $E^X_G$ is Borel and satisfies (C).
As a corollary we get
$$\operatorname{(A)} \Leftrightarrow \operatorname{(D)} \Leftrightarrow  \operatorname{(E)} \Rightarrow  \operatorname{(C)} \Rightarrow  \operatorname{(B)}$$
without assuming that $E^X_G$ is Borel.

\section{Application}\label{sec:Applic}

Let $G$ be a Polish group and $X$ be a Polish $G$-space.
Suppose that $E^X_G$ satisfies (A).
It is natural to ask if there is a connection between properties of $G$ and the position of $E^X_G$ in the Borel reducibility among countable Borel equivalence relations.
For example, we say that $E^X_G$ is {\it essentially hyperfinite} if $E^X_G\sim_B F$ where $F$ is a countable Borel equivalence relation induced by a Borel action of $\mathbb{Z}$.
Variations of the following question appeared in \cite[Conjecture~8.4]{Ding Gao} or \cite[Question~5.7.5]{Kanovei}.

\begin{question}\label{question:main}
Let $E^X_G$ be an essentially countable orbit equivalence relation induced by a continuous action of an {\bf abelian} Polish group $G$ on a Polish space $X$.
Is it true that $E^X_G$ is essentially hyperfinite?
\end{question}

The answer is affirmative in the case when the abelian Polish group $G$ is discrete, see~\cite{Gao Jackson}, non-archimedean, see~\cite{Ding Gao}, and locally compact, see~\cite{Cotton}.

Next result is derived directly from Theorem~\ref{th:main1}, we note that it is a variation on \cite[Theorem~7.3]{Hjorth Kec}.

\begin{theorem}\label{th:app}
Let $G$ be a non-archimedean Polish group that admits two-sided invariant metric and $X$ be a Polish $G$-space.
Suppose that $E^X_G$ is essentially countable.
Then there is a sequence of open normal subgroups $\{N_n\}_{n\in \mathbb{N}}$ and a continuous actions $H_n=G/N_n \curvearrowright X_n$ where $X_n$ is a Polish space such that 
$$E^X_G\le_B \bigoplus_{n\in \mathbb{N}} E^{X_n}_{H_n}.$$ 
\end{theorem}

First we need a variation of an unpublished result of Conley and Dufloux, see~\cite[Theorem~3.11]{Kec CTBL}.
They considered locally compact groups but not necessarily with two-sided invariant metric. 

\begin{lemma}\label{lm:tech}
Let $G$ be a Polish group that admits a two-sided invariant metric and $X$ be a Polish $G$-space such that $E^X_G$ is a Borel equivalence relation.
Let $V\subseteq G$ be an open symmetric conjugacy-invariant neighbourhood of $1_G$ and $B\subseteq X$ be a Borel $V$-lacunary complete section of $E^X_G$, i.e., $g\cdot x=y$ for $g\in V$ and $x,y\in B$ implies that $x=y$.
Then there is $C\supseteq B$ a Borel $V$-lacunary complete section of $E^X_G$ such that $V^2\cdot C=X$.
\end{lemma}
\begin{proof}
Fix some countable dense subset $\{g_i\}_{i\in \mathbb{N}}\subseteq G$.
Let $C_0=B$ and define inductively $C_{i+1}=C_i\cup (g_i\cdot C_i\setminus V\cdot C_i)$.
We claim that $C=\bigcup_{i\in \mathbb{N}}C_i$ works as required.

First note that $C_i$ is a countable section for every $i\in \mathbb{N}$.
We show by induction that $C_i$ is Borel for every $i\in \mathbb{N}$.
If $i=0$, then it follows from the assumption on $B$.
Suppose that $C_i$ is Borel.
\begin{claim*}
The set $T_i=\{(g\cdot x,x)\in X\times X:x\in C_i \ \& \ g\in V\}$ is Borel .
\end{claim*}
\begin{proof}
The assumption that $E^X_G$ is Borel together with \cite[Theorem~7.1.2]{Becker Kec} gives that the assignment $y\mapsto \operatorname{stab}(y)=\{g\in G:g\cdot y=y\}$ is Borel.
Note that $\operatorname{stab}(y)$ is non-empty closed subset of $G$ and by \cite[Theorem~12.13]{Kec} there is a Borel map $y\mapsto (g_{j,y})_{j\in \mathbb{N}}$ such that $(g_{j,y})_{j\in \mathbb{N}}$ is dense subset of $\operatorname{stab}(y)$ for every $y\in X$.
We claim that the relation
$$R_V=\{(y,x)\in X\times X: \exists g\in V \ g\cdot x=y\}=\{(g\cdot x,x)\in X\times X: x\in X \ \& \ g\in V\}$$
is Borel.
It is clearly analytic by the definition.
We show that the complement is analytic as well, we have
$$(y,x)\not \in R_V \ \Leftrightarrow \ (y,x)\not \in E^X_G \ \vee (\exists h\in G \ h\cdot x=y \ \wedge \ \forall j\in \mathbb{N} \ g_{j,y}\cdot h\not \in V).$$
Finally, note that $T_i=X\times C_i \cap R_V$.
\end{proof}
It is easy to see that $T_i$ has countable vertical sections.
This is because $(T_i)_y=\{x\in X:(y,x)\in T_i\}\subseteq C_i\cap [y]_{E^X_G}$ for every $y\in X$ and we know that $C_i$ is a countable section.
By Lusin--Novikov Theorem~\cite[Theorem~18.10]{Kec} we have that $V\cdot C_i$, which is equal to the projection of $T_i$ to the first coordinate, is a Borel set and so is the set $C_{i+1}$.
This gives immediately that $C$ is Borel.

Suppose that $x\in X$.
Then there is $y\in C_0$, $h\in G$ and $i\in \mathbb{N}$ such that $h\cdot y=x$ and $h^{-1}\cdot g_i\in V$.
Let $z=g_i\cdot y$.
Then either $z\in C_{i+1}$ and therefore $x\in V\cdot z$, or there is $z_0\in C_i$ such that $z\in V\cdot z_0$ and then we have $x\in V^2\cdot z_0$.
This shows that $X=V^2\cdot C$.

It remains to show that $C$ is $V$-lacunary.
We show by induction that $C_i$ is $V$-lacunary for every $i\in \mathbb{N}$.
It clearly holds for $i=0$.
Let $x,y\in C_{i+1}$ and suppose that $y\in V\cdot x$.
If $x,y\not\in C_i$, then there is $x_0,y_0\in C_i$ such that $g_i\cdot x_0=x$ and $g_i\cdot y_0=y$.
Then we have $y_0\in g^{-1}_i\cdot V\cdot g_i \cdot x_0=V\cdot x_0$ because $V$ is conjugacy invariant and therefore $x=y$ by the inductive assumption.
If $x\in C_i$, then $y\in C_i$ by the definition of $C_{i+1}$.
Again, the inductive assumption gives $x=y$ and that finishes the proof.
\end{proof}

\begin{proof}[Proof of Theorem~\ref{th:app}]
Using Theorem~\ref{th:main1} we get sequences $\{B_n\}_{n\in \mathbb{N}}$ and $\{V_n\}_{n\in \mathbb{N}}$ where $B_n$ is $V_n$-lacunary Borel section and $V_n$ is an open neighbourhood of identity.
By the assumption on $G$ we find $N_n\subseteq V_n$ an open normal subgroup and by Lemma~\ref{lm:tech} $X_n\supseteq B_n$ an $N_n$-lacunary Borel section such that $N^2_n\cdot X_n=N_n\cdot X_n=[B_n]_{E^X_G}$ for every $n\in \mathbb{N}$.

For each $n\in \mathbb{N}$, $(gN_n)\in H_n=G/N_n$ and $x\in X_n$ define $(gN_n)\star x$ to be the unique element of $X_n$ in $g\cdot N_n\cdot x$.
It follows from the maximality and $N_n$-lacunarity of $X_n$ that this is a well-defined map and it can be easily verified that it is an action $H_n\curvearrowright X_n$.
Moreover, it follows from Lusin--Novikov Theorem~\cite[Theorem~18.10]{Kec} that the action is Borel for every $n\in \mathbb{N}$.
Another use of Lusin--Novikov Theorem~\cite[Theorem~18.10]{Kec} gives the desired reduction.
\end{proof} 

\begin{corollary}\cite{Ding Gao}
Let $G$ be an abelian non-archimedean Polish group and $X$ be a Polish $G$-space such that $E^X_G$ is essentially countable.
Then $E^X_G$ is essentially hyperfinite.
\end{corollary}
\begin{proof}
This is a combination of Theorem~\ref{th:app} and \cite[Corollary~8.2]{Gao Jackson}.
\end{proof}

\section{Proof of Theorem~\ref{th:main1}}\label{sec:main}

Let $X$ be a Polish space and $F$ an equivalence relation on $X$.
We denote as $[x]_F$ the $F$-equivalence class of $x\in X$.
Let $G$ be a Polish group that acts continuously on a Polish space $X$ and $E^X_G$ be the corresponding orbit equivalence relation on $X$.
We denote the $\sigma$-ideal of meager subsets of $G$ as $\mathcal{M}_G$.
For $A\subseteq X$ we define $G(x,A)=\{g\in G:g\cdot x\in A\}$ and $E^A_G=E^X_G\upharpoonright A\times A$.

We say that $C\subseteq X$ is a {\it $G$-lg ($G$-locally globally) comeager} set if $G\setminus G(x,C)\in \mathcal{M}_G$ for every $x\in X$.
Using the category quantifier $\forall^*$, for comeager many, this can be equivalently stated as
$$\forall x\in X  \  \forall^* g\in G \ (g\cdot x\in C).$$
Note that the collection of $G$-lg comeager sets is closed under supersets and countable intersections.
If $G$ is countable, then the only $G$-lg comeager set is $X$.
Even though we do not need it here, we remark that it follows from \cite[Theorem~8.41]{Kec} that if $C\subseteq X$ is a Borel $G$-lg comeager set, then $C$ is comeager in $X$.
This might serve as an explanation for the word ``globally'' in the definition.
More generally, a Borel set $C\subseteq X$ is $G$-lg comeager if and only if it is comeager in every finer Polish topology on $X$ such that the action of $G$ is continuous.

Next we collect the technical statements that we need in the proof.

\begin{proposition}\label{pr:0}
Let $C\subseteq X$ be a Borel $G$-lg comeager set.
Then $E^C_G\sim_B E^X_G$.
\end{proposition}

\begin{proposition}\label{pr:1}
Let $F\subseteq E^X_G$ be a Borel equivalence relation on $X$ such that each $E^X_G$-class contains at most countably many $F$-classes.
Then there is a Borel $G$-lg comeager set $C\subseteq X$ such that $G(x,C\cap [x]_F)$ is relatively open in $G(x,C)$ for every $x\in C$, i.e., for every $x\in C$ there is $V\subseteq G$ open neighbourhood of $1_G$ such that $V\cdot x\cap C\subseteq [x]_F\cap C$.
\end{proposition}

\begin{proposition}\label{pr:2}
Let $F\subseteq E^X_G$ be a Borel equivalence relation on $X$ such that each $E^X_G$-class contains at most countably many $F$-classes.
Then $E^X_G$ is Borel.
\end{proposition}

\begin{proposition}\cite[Theorem~18.6]{Kec}\cite[Theorem~18.6*]{Kec2}\cite[Proof of Lemma~3.7]{Kec Tarski}\label{pr:uni}
Let $Y,X$ be standard Borel spaces and $P\subseteq Y\times X$ be Borel with $A=\operatorname{proj}_Y(P)$.
Let $y\in A\mapsto I_y$ be a map assigning to each $y\in A$ a $\sigma$-ideal of subsets of $P_y$ such that:
\begin{enumerate}
	\item [(i)] For each Borel $R\subseteq P$, there is a ${\bf \Sigma}^1_1$ set $S\subseteq Y$ and a ${\bf \Pi}^1_1$ set $T\subseteq Y$ such that 
	$$y\in A \ \Rightarrow \ [R_y\in I_y \ \Leftrightarrow \ y\in S \ \Leftrightarrow \ y\in T],$$
	\item [(ii)] $y\in A \Rightarrow \ P_y\not\in I_y$.
\end{enumerate}
Then there is a Borel uniformization of $P$ and, in particular, $A$ is Borel.
\end{proposition}

\begin{proof}[Proof of Theorem~\ref{th:main1}]
Suppose that $E^X_G$ satisfies (E).
We show that $E^X_G$ is Borel and satisfies (C).

Let $\varphi:X\to Y$ be as in (E).
Define $F=(\varphi^{-1}\times \varphi^{-1})(=_Y)$, i.e., $(x,y)\in F$ if and only if $\varphi(x)=\varphi(y)$.
Then it follows from (E) that $F$ is a Borel equivalence relation and every $E^X_G$-class contains at most countably many $F$-classes.
By Proposition~\ref{pr:2} we have that $E^X_G$ is Borel and by Proposition~\ref{pr:1} we find a Borel $G$-lg comeager set $C\subseteq X$ such that $G(x,C\cap [x]_F)$ is relatively open in $G(x,C)$ for every $x\in C$.

Next we want to apply Proposition~\ref{pr:uni}.
Define $P\subseteq Y\times X$ as $P=\{(\varphi(x),x):x\in C\}$, $A=\operatorname{proj}_Y(P)$ and the assignment $\varphi(x)\in A\mapsto I_{\varphi(x)}$ as 
$$B\in I_{\varphi(x)} \ \Leftrightarrow \ G(x,B\cap C\cap [x]_F)\in \mathcal{M}_G$$
where $x\in C$ and $B\subseteq C\cap [x]_F$.

We verify the assumptions of Proposition~\ref{pr:uni}.
It is easy to see that $P$ is a Borel set because it is just the reversed graph of the Borel function $\varphi\upharpoonright C:C\to Y$.
Let $x,y\in C$ such that $(x,y)\in F$, i.e., $\varphi(x)=\varphi(y)$.
Especially, there is $g\in G$ such that $g\cdot x=y$.
Let $B\subseteq C\cap [x]_F$.
Note that $G(y,B\cap C\cap [x]_F)\cdot g=G(x,B\cap C\cap [x]_F)$.
This implies that the assignment $\varphi(x)\in A\mapsto I_{\varphi(x)}$ is well-defined and it is easy to see that $I_{\varphi(x)}$ is an $\sigma$-ideal of subsets of $P_{\varphi(x)}$ for every $x\in C$.
Moreover, since $V\cdot x\cap C\subseteq C\cap [x]_F=P_{\varphi(x)}$ for some open set $1_G\in V\subseteq G$ and $C$ is $G$-lg comeager we have that $P_{\varphi(x)}\not\in I_{\varphi(x)}$ for every $x\in C$.
It remains to show that (ii) in Proposition~\ref{pr:uni} holds as well.
To this end pick a Borel set $R\subseteq P$.
Define the set $R'$ as
$$R'=\{(r,s)\in X\times X:r,s\in C \ \& \ (\varphi(r),s)\in R\}.$$
Note that $R'$ is Borel and we have $R'_{r_0}=R'_{r_1}$ whenever $r_0,r_1\in C$ such that $(r_0,r_1)\in F$.
Then for $r\in C$ we have
\begin{equation}
\tag{*}
R_{\varphi(r)}\in I_{\varphi(r)} \ \Leftrightarrow \ G(r,R_{\varphi(r)}\cap C\cap [r]_F)\in \mathcal{M}_G \ \Leftrightarrow \ G(r,R'_r)\in \mathcal{M}_G
\label{star}
\end{equation}
because $R_{\varphi(r)}=R'_r\subseteq C\cap [r]_F$.
It follows from \cite[Theorem~16.1]{Kec} together with \eqref{star} that the sets
$$\mathcal{Z}_0=\{r\in C:G(r,R'_r)\in \mathcal{M}_G\} \ \& \ \mathcal{Z}_1=\{r\in C:G(r,R'_r)\not\in \mathcal{M}_G\}$$
are Borel and $F\upharpoonright C\times C$-invariant.
Set $S=\varphi(\mathcal{Z}_0)$ and $T=Y\setminus \varphi(\mathcal{Z}_1)$.
Then $S\subseteq Y$ is ${\bf \Sigma}^1_1$ and $T\subseteq Y$ is ${\bf \Pi}^1_1$ because $\varphi$ is a Borel map and the rest follows again from \eqref{star}.

Having verified the assumptions of Proposition~\ref{pr:uni}, we get that the set $A$ is Borel and there is a Borel map $f:A\to C$ such that $(y,f(y))\in P$ for every $y\in A$.
It is easy to see that $(f(y),f(z))\not \in F$ for every $y\not=z\in A$ because $\varphi(f(y))=y$ for every $y\in A$ by the definition of $P$.
Especially, $f$ is injective and $\varphi\circ f:A\to A$ is the identity on $A$.
It follows that $D=f(A)\subseteq C$ is a Borel countable complete section of $E^X_G$ and a transversal of the equivalence relation $F\upharpoonright C\times C$ on $C$.

Pick any decreasing sequence $\{V_n\}_{n\in \mathbb{N}}$ of open neighbourhoods of $1_G$ such that $\{1_G\}=\bigcap_{n\in \mathbb{N}} V_n$.
Define
$$B_n=\{x\in D:V_n\cdot x\cap C\subseteq [x]_F\}.$$
We claim that $\{B_n\}_{n\in \mathbb{N}}$ and $\{V_n\}_{n\in \mathbb{N}}$ is the sequence from (C).
It follows from the fact that $G(x,C\cap [x]_F)$ is relatively open in $G(x,C)$ for every $x\in C$ together with the fact that $D\subseteq C$ that $D=\bigcup_{n\in \mathbb{N}}B_n$.
The definition of $B_n$ together with the fact that $D$ is a transversal of $F\upharpoonright C\times C$ implies that if $g\cdot x=y$ for some $g\in V_n$ and $x,y\in B_n$, then $x=y$.
It remains to show that $B_n$ is Borel for every $n\in \mathbb{N}$.
To see this first note that the set
$$C_n=\{(x,g)\in D\times V_n:(x,g\cdot x)\in F\}$$
is Borel because $F$ and $D$ are Borel sets.
Then we have
$$B_n=\{x\in D:(\forall^*g\in V_n) (x,g)\in C_n\}$$
and the set on the right-hand side is Borel by \cite[Theorem~16.1]{Kec}.
This finishes the proof.
\end{proof}

\section{Technical Proofs}

\begin{proof}[Proof of Proposition~\ref{pr:0}]
Define $D=\{(x,g)\in X\times G:g\cdot x\in C\}$.
Then $D$ is a Borel set, $\operatorname{proj}_X(D)=X$ and $D_x\not\in \mathcal{M}_G$ for every $x\in X$.
By \cite[Theorem~18.6]{Kec} or Proposition~\ref{pr:uni} there is a Borel function $f:X\to G$ such that $(x,f(x))\in D$ for every $x\in X$.
The function
$$F(x)=f(x)\cdot x$$
is the desired Borel reduction from $E^X_G$ to $E^C_G$. 
\end{proof}

\begin{proof}[Proof of Proposition~\ref{pr:1}]
Let $\{V_n\}_{n\in \mathbb{N}}$ be an open basis at $1_G$ made of symmetric sets such that $V_{n+1}\cdot V_{n+1}\subseteq V_n$.
Define
$$C=\{x\in X:(\exists n\in \mathbb{N})(\forall^* g\in V_n) \ (x,g\cdot x)\in F\}.$$
It follows from \cite[Theorem~16.1]{Kec} that $C$ is a Borel set.
Let $x\in C$ and $n\in \mathbb{N}$ such that $(x,g\cdot x)\in F$ for comeager many $g\in V_n$.
Take $g\in G(x,C)\cap V_{n+1}$.
Then we have that $V_{n+1}\cdot g\subseteq V_n$ and therefore $(x,h\cdot g\cdot x)\in F$ for comeager many $h\in V_{n+1}$.
By the choice of $g$ we have that $g\cdot x\in C$ and by the definition of $C$ we find $n'\in \mathbb{N}$ such that $(g\cdot x,h'\cdot g\cdot x)\in F$ for comeager many $h'\in V_{n'}$.
This shows that $(x,g\cdot x)\in F$ and as a consequence that $G(x,C\cap [x]_F)$ is relatively open in $G(x,C)$.

It remains to show that $C$ is $G$-lg comeager in $G$.
Suppose that there is $x\in X$ such that $G(x,C)$ is not comeager.
By \cite[Theorem~8.26]{Kec} there is an open set $U\subseteq G$ such that $G(x,C)$ is meager in $U$.
Let $\{\mathfrak{f}_{i}\}_{i\in \mathbb{N}}$ be an enumeration of the $F$-classes that are subset of $[x]_{E^X_G}$.
Define $D_i=G(x,\mathfrak{f}_i)$.
It follows that $D_i$ has the Baire property for every $i\in \mathbb{N}$ and that $U\subseteq \bigcup_{i\in \mathbb{N}} D_i$.
Another use of \cite[Theorem~8.26]{Kec} gives an open set $V\subseteq U$ and $i\in \mathbb{N}$ such that $D_i$ is comeager in $V$.
In another words $h\cdot x\in \mathfrak{f}_i$ for comeager many $h\in V$.
Pick $g\in (V\cap D_i)\setminus G(x,C)$ and $n\in \mathbb{N}$ such that $V_n\cdot g\subseteq V$.
Then we have that $D_i$ is comeager in $V_n\cdot g$ and $g\cdot x\in \mathfrak{f}_i$.
This shows that there are comeager many $h\in V_n$ such that $(g\cdot x,h\cdot g\cdot x)\in F$.
That is a contradiction with $g\not\in G(x,C)$ and that finishes the proof.
\end{proof}

\begin{proof}[Proof of Proposition~\ref{pr:2}]
Let $C\subseteq X$ be as in Proposition~\ref{pr:1}.
We claim that
\begin{equation}\label{eq:Countable classes implies Borel}
(x,y)\in E^C_G \ \Leftrightarrow \ (\exists^* (a,b)\in G\times G) \ (a\cdot x,b\cdot y)\in F
\end{equation}
for every $x,y\in C$.

Let $x,y\in C$.
If $x,y$ satisfies the right-hand side of \eqref{eq:Countable classes implies Borel}, then $(x,y)\in E^C_G$ because $F\subseteq E^X_G$.
Suppose, on the other hand, that $(x,y)\in E^C_G$.
By the definition of $E^C_G$ and $C$ we find an open set $1_G\in V\subseteq G$ and $g\in G$ such that $g\cdot x=y$ and $V\cdot y\cap C\subseteq [y]_F\cap C$.
Note that $W=G(y,V\cdot y\cap C)=V\cap G(y,C)$ is nonmeager and $a\cdot y\in [y]_F$ for every $a\in W$.
The set $W\cdot g\times W$ is nonmeager in $G\times G$.
Let $(a\cdot g,b)\in W\cdot g\times W$.
Then we have $a\cdot g\cdot x=a\cdot y\in [y]_F\cap C$ and $b\cdot y\in [y]_F\cap C$ by the definition of $W$.
This shows that $x,y$ satisfies the right-hand side of \eqref{eq:Countable classes implies Borel}.

It remains to show that the right-hand side of \eqref{eq:Countable classes implies Borel} defines a Borel set.
The set
$$R=\{(x,y,g,h)\in C\times C\times G\times G:(g\cdot x,h\cdot y)\in F\}$$
is Borel because $F$ is a Borel equivalence relation and $C$ is a Borel set.
This implies by \cite[Theorem~16.1]{Kec} that $E^C_G$ is a Borel equivalence relation and Proposition~\ref{pr:0} finishes the proof.
\end{proof}

\end{document}